\documentclass[11pt]{amsart}
\usepackage{mathrsfs, latexsym,amsfonts,amssymb}

\newtheorem{theorem}{Theorem}[section]
\newtheorem{lemma}[theorem]{Lemma}
\newtheorem{corollary}[theorem]{Corollary}
\newtheorem{question}[theorem]{Question}
\theoremstyle{definition}
\newtheorem{definition}[theorem]{Definition}
\newtheorem{proposition}[theorem]{Proposition}
\theoremstyle{remark}

\newtheorem{example}[theorem]{Example}

\begin{document}
\noindent \vspace{0.5in}

\title[Some topological properties of Charming spaces]
{Some topological properties of Charming spaces}

\author{Xiaoting Li}
\address{(Xiaoting Li): School of mathematics and statistics,
Minnan Normal University, Zhangzhou 363000, P. R. China}

\author{Fucai Lin}
\address{(Fucai Lin): School of mathematics and statistics,
Minnan Normal University, Zhangzhou 363000, P. R. China}
\email{linfucai2008@aliyun.com; lfc19791001@163.com}

\author{Shou Lin}
\address{(Shou Lin): Institute of Mathematics, Ningde Teachers' College, Ningde, Fujian
352100, P. R. China}
\email{shoulin60@163.com}
\thanks{The authors are supported by the NSFC (Nos. 11201414, 11471153, 11571158), the Natural Science Foundation of Fujian Province (No. 2012J05013) of China and Training Programme Foundation for Excellent Youth Researching
Talents of Fujian's Universities (JA13190).}


\begin{abstract}
In this paper, we mainly discuss the class of charming spaces, which was introduced by A.V. Arhangel'skii in [Remainders of metrizable spaces and a generalization of Lindel\"{o}f $\Sigma$-spaces,
Fund. Math., 215(2011), 87-100].  We first show that there exists a charming space $X$ such that $X^{2}$ is not a charming space. Moreover, we discuss some properties of charming spaces and give some characterizations of some special class of charming spaces. Finally, we show that the Suslin number of an arbitrary charming rectifiable space is countable.
\end{abstract}
\maketitle

\section{Introduction}
All the spaces are Tychonoff unless stated otherwise. Readers may refer to
\cite{E1989, Gr} for notations and terminology not
explicitly given here.

In 1969, Nagami in \cite{K1969} introduced the notion of  $\Sigma$-spaces, and then the class of $\Sigma$-spaces with the Lindel\"{o}f property (i.e. the class of Lindel\"{o}f $\Sigma$-spaces) quickly attracted the attention of some topologists. From then on, the study of Lindel\"{o}f $\Sigma$-spaces has become an important part in the functional analysis, topological algebra and descriptive set theory. V.V.Tkachuk in \cite{T2010} detailly described the Lindel\"{o}f $\Sigma$-spaces and  made an overview of the recent progress achieved in the study of Lindel\"{o}f $\Sigma$-spaces. Recently, A.V.Arhangel'skii has proved the following theorem when he studied the properties of remainders of metrizable spaces.

\begin{theorem}\cite{A0}
Let $X$ be a metrizable space. If the weight of $X $does not exceed $2^\omega$, then any remainders of $X$ in a Hausdorff compactification is a Lindel\"{o}f $\Sigma$-space.
\end{theorem}

It is natural to ask if we can find a class of spaces $\mathscr{P}$ such that each remainder of a Hausdorff compactification of an arbitrary metrizable space belongs to $\mathscr{P}$. Therefore, A.V.Arhangel'skii defined the following concept of charming spaces.

\begin{definition}\cite{A0}
Let $X$ be a space. If there exists a Lindel\"{o}f $\Sigma$-subspace $Y$ such that for each open neighborhood $U$ of $Y$ in $X$ the set $X\setminus U$ is also a Lindel\"{o}f $\Sigma$-subspace, then we say that $X$ is a {\it charming space}.
\end{definition}

Then he showed the following theorem.

\begin{theorem}\cite{A0}
Let $X$ be a paracompact $p$-space. Then any remainder of $X$ in a Hausdorff compactification is a charming space.
\end{theorem}

Indeed, A.V.Arhangel'skii defined many new classes of spaces (that is, ($i, j$)-structured spaces) which have similar structures with the class of charming spaces, and he said that each of the classes of spaces so defined is worth studying. Therefore, we mainly discuss some topological properties of ($i, j$)-structured spaces.

\section{Preliminaries}
Let $X$ be a space, and let $\mathscr{N}$ be a family of subsets of a space $X$. Then the family $\mathscr{N}$ is a {\it network} of $X$ if every open subset $U$ is the union of some subfamily of $\mathscr{N}$.
We say that $X$ is a {\it Lindel\"{o}f $p$-space} if it is the preimage of a separable metrizable space under a perfect map; $X$ is a
{\it Lindel\"{o}f $\Sigma$-space} if there exists a space $Y$ which maps continuously onto $X$ and perfectly onto a second countable space, that is, a Lindel\"{o}f $\Sigma$-space is the continuous onto image of some Lindel\"{o}f $p$-space. Therefore, a Lindel\"{o}f $p$-space is a Lindel\"{o}f $\Sigma$-space. It is well-known that the class of Lindel\"{o}f $\Sigma$-spaces contains the classes of $\sigma$-compact spaces and spaces with a countable network.

Let $\mathscr{P}$ and $\mathscr{Q}$ be two classes of topological spaces respectively. A space $X$ will be called $(\mathscr{P}, \mathscr{Q})$-{\it structured} \cite{A0} if there is a subspace $Y$ of $X$ such that $Y\in\mathscr{P}$, and for each open neighborhood $U$ of $Y$ in $X$, the subspace $X\setminus U$ of $X$ belongs to $\mathscr{Q}$. In this situation, we call $Y$ a $(\mathscr{P}, \mathscr{Q})$-{\it shell} \cite{A0} of the space $X$.

Throughout this paper, we always let $\mathscr{P}_{0}$ be the class of $\sigma$-compact spaces, $\mathscr{P}_{1}$ be the class of separable metrizable spaces, $\mathscr{P}_{2}$ be the class of spaces with a countable network,  $\mathscr{P}_{3}$ be the class of Lindel\"{o}f $p$-spaces, $\mathscr{P}_{4}$ be the class of Lindel\"{o}f $\Sigma$-spaces and $\mathscr{P}_{5}$ be the class of compact spaces. For $i, j\in\{0, 1, 2, 3, 4, 5\}$, a space $X$ will be called $(i, j)$-{\it structured} \cite{A0} if it is $(\mathscr{P}_{i}, \mathscr{P}_{j})$-structured. In particular, a $(4, 4)$-structured space is called a {\it charming space}.

\section{The topological properties of $(i, j)$-structured spaces}
In this section, we shall discuss some topological properties of $(i, j)$-structured spaces. We first consider the closed hereditary property of $(i, j)$-structured spaces.

\begin{proposition}\label{p0}
Each closed subspace of an $(i, j)$-structured space $X$ is $(i, j)$-structured, where $i, j\in\{0, 1, 2, 3, 4, 5\}$.
\end{proposition}
\begin{proof}
Let $A$ be a closed subspace of $(i,j)$-structured space $X$. Since $X$ is an $(i,j)$-structured space, there is a subspace $Y$ such that $Y$ is a $(\mathscr{P}_{i}, \mathscr{P}_{j})$-shell of $X$. Let $B=A \cap Y$. We claim that the subspace $B$ is a $(\mathscr{P}_{i}, \mathscr{P}_{j})$-shell of the subspace $A$. Obviously, we have $B\in\mathscr{P}_{i}$. Now it suffices to show that for each open neighborhood $V$ of $B$ in $A$, the subspace $A\setminus V$ of $A$ belongs to $\mathscr{P}_{j}$. Indeed, let $V$ be an open neighborhood of $B$ in $A$. Then there is an open neighborhood $V_1$ in $X$ such that $V=V_1 \cap\ A$ and $A\setminus V=A\setminus V_1$. Hence, we have $Y\setminus V_1\subset X\setminus A$ and $V_1\cap Y\subset V_1$, which implies $(Y\setminus V_1)\cup(V_1\cap Y )=Y\subset V_1\cup(X\setminus A)$. Now it is easy to see that $A\setminus V=A\setminus V_1\subset X\setminus(V_1\cup(X\setminus A))\in\mathscr {P}_{j}$.
\end{proof}

The following three propositions are easy exercises.

\begin{proposition}\label{pcontinuous}
Any image of an $(i, j)$-structured space under a continuous
map is an $(i, j)$-structured space, where $i, j\in\{0, 2, 4, 5\}$.
\end{proposition}

\begin{proposition}
Any preimage of an $(i, j)$-structured space under a perfect
map is an $(i, j)$-structured space, where $i, j\in\{0, 3, 4, 5\}$.
\end{proposition}

\begin{proposition}\label{pa}
For $i\in \{0, 2, 4\}$, if $X_{j}\in \mathscr{P}_{i}$ for each $j\in\mathbb{N}$ and $X=\bigcup_{j\in\mathbb{N}}X_{j}$, then $X\in \mathscr{P}_{i}$.
\end{proposition}

However, the following questions are still unknow for us.

\begin{question}
For arbitrary $i, j\in\{1, 3\}$, is any image of an $(i, j)$-structured space under a continuous
map an $(i, j)$-structured space?
\end{question}

\begin{question}
For arbitrary $i, j\in\{1, 2\}$, is any preimage of an $(i, j)$-structured space under a perfect
map an $(i, j)$-structured space?
\end{question}

\begin{proposition}\label{p1}
Let $X=\bigcup_{k\in\omega} X_k$, where each $X_k$ is an $(i, j)$-structured space. Then $X$ is also an $(i, j)$-structured space, where $i, j\in \{0, 2, 4\}$.
\end{proposition}
\begin{proof}
Fix $i, j\in \{0, 2, 4\}$. For every $k\in\omega$, since $X_k$ is an $(i, j)$-structured space, there exists a subspace $Y_k\subset X_k$ such that the subspace $Y_{k}$ is a $(\mathscr{P}_{i}, \mathscr{P}_{j})$-shell of the space $X_{k}$. Put $Y=\bigcup_{k\in\omega} Y_k$. We claim that $Y$ is a $(\mathscr{P}_{i}, \mathscr{P}_{j})$-shell of the space $X$. Indeed, the space $Y$ belongs to $\mathscr{P}_{i}$ by Proposition~\ref{pa}. Moreover, for each open neighborhood $U$ of $Y$ in $X$, each $X_k\setminus (X_k\cap U)\in \mathscr{P}_{j}$, hence it follows from Proposition~\ref{pa} that $$X\setminus U=(\bigcup_{k\in\omega} X_k)\setminus U=\bigcup_{k\in\omega}(X_k\setminus U)=\bigcup _{k\in\omega}(X_k\setminus X_k\cap U)\in \mathscr{P}_{j}.$$ Therefore, $X$ is an $(i, j)$-structured space.
\end{proof}

Since the intersection of countably many Lindel\"{o}f $\Sigma$-subspaces is also a Lindel\"{o}f $\Sigma$-space, it is natural to pose the following question.

\begin{question}\label{q001}
Let $X$ be a space and $X_k\subset X$ for each $k\in\mathbb{N}$. If each $X_k$ is an $(i, j)$-structured space, is $\bigcap_{k\in\mathbb{N}}X_{k}$ an $(i, j)$-structured space, where $i, j\in \{0, 1, 2, 3, 4, 5\}$?
\end{question}

It is well-known that the product of countably many Lindel\"{o}f $\Sigma$-spaces is also a Lindel\"{o}f $\Sigma$-space. However, the product of two charming spaces may not be a charming space, see Example~\ref{e1}.

We know that each discrete space $X$ has a Hausdorff one point Lindel\"{o}fication which defined as follows: take an arbitrary point $p\not\in X$ and consider the set $Y=X\cup \{p\}$, and then let each point of $X$ be open and each neighborhood of $p$ be the form $U\cup\{p\}$, where $U$ is open in $X$ and $X\setminus U$ is a countable set.

\begin{example}\label{e1}
There exists a charming space $X$ satisfying the following conditions:
\begin{enumerate}
\item $X$ is not a Lindel\"{o}f $\Sigma$-space.

\item The product of $X^{2}$ is not a charming space.
\end{enumerate}
\end{example}

\begin{proof}
Let $X=\{\infty\}\cup D$ be the one-point Lindel\"{o}fication of an uncountable discrete space, where $D$ is an uncountable discrete space. V.V.Tkachuk \cite{T2010} has proved that $X$ is not a Lindel\"{o}f $\Sigma$-space.
Obviously, the subspace $\{\infty\}$
 is a Lindel\"{o}f $\Sigma$-space, and for each open neighborhood $V$ of $\{\infty\}$ in $X$ we have $D\setminus V$ is Lindel\"{o}f, then $D\setminus V$ is separable and metrizable. Thus $D\setminus V$ is a Lindel\"{o}f $\Sigma$-space. Therefore, $X$ is a Charming space.

Next we shall show that $X^2$ is not a charming space. Assume that $X^2$ is a charming space. Then there exists a Lindel\"{o}f $\Sigma$-subspace $L\subset X^{2}$ such that, for each open neighborhood $U$ of $L$, $X^{2}\setminus U$ is a Lindel\"{o}f $\Sigma$-subspace. We claim that $(\infty,\infty)\in L$. 
Assume on the contrary that $(\infty,\infty)\not\in L$. Let $D_{1}=\pi_{1}(L)\cap D$ and $D_{2}=\pi_{2}(L)\cap D$, where $\pi_{1}$ and $\pi_{2}$ are the projections to the first coordinate and the second coordinate respectively. Then it follows from the above facts that $D_{1}$ and $D_{2}$ must be countable. Therefore, $L$ is a countable set. Therefore, it is easy to see that there exists a neighborhood $V$ of $L$ such that $X^2\setminus V$ is uncountable. However, since $X^2\setminus V$ is a Lindel\"{o}f $\Sigma$-subspace, we can similarly obtain that $X^2\setminus V$ is countable, which is a contradiction.

Since each compact subset of $X^{2}$ is finite, it follows from \cite[Theorem 5]{T2010} that $L$ must be a countable set. Then there exists a point $x_{0}\in D$ such that $(\{x_{0}\}\times X)\cap L=\emptyset$ and $(X\times\{x_{0}\})\cap L=\emptyset$. Then it is easy to see that $((X\setminus\{x_0\})\times (X\setminus\{x_0\}))$ is an open neighborhood of $L$ in $X^{2}$. Put $V=X\setminus\{x_0\}$. Then $$X^2\setminus (V\times V)=\{(x_0,x_0)\}\cup(\{x_0\}\times V)\cup(V\times\{x_0\})$$ is a Lindel\"{o}f $\Sigma$-space. Since $\{x_0\}\times V$ is closed in $X^2\setminus (V\times V)$, we see that $\{x_0\}\times V$ is a Lindel\"{o}f $\Sigma$-space. However, $V$ is homeomorphic to $X$, which is a contradiction. Therefore, $X^2$ is not a charming space.
\end{proof}

 {\bf Remark} (1) In \cite{T2010}, V.V.Tkachuk proved that if every compact subset of a Lindel\"{o}f $\Sigma$-space $X$ is finite then $X$ is countable. Obviously, we can not replace `` Lindel\"{o}f $\Sigma$-spaces'' with ``Charming spaces''in Example~\ref{e1}.

 (2) The product of a family of finitely many Charming spaces need not be a Charming space.

(3) The space $X$ in Example~\ref{e1} is $(i, j)$-structured for $i, j\in \{0, 1, 2, 3, 4\}$. Therefore, the product of a family of finitely many $(i, j)$-structured spaces need not be an $(i, j)$-structured space.

It is easy to check that $X^{2}$ in Example~\ref{e1} is a Lindel\"{o}f space. Therefore, we have the following question:

\begin{question}\label{q0001}
Is the product of two charming spaces Lindel\"{o}f? In particular, is the product of a charming space with a Lindel\"{o}f $\Sigma$-space a Lindel\"{o}f space?
\end{question}

Finally, we discuss a charming space with a $G_{\delta}$-diagonal. The following theorem is well known in the class of Lindel\"{o}f $\Sigma$-spaces.

\begin{theorem}\cite{T2010}\label{t125}
A Lindel\"{o}f $\Sigma$-space with a $G_{\delta}$-diagonal has a countable network.
\end{theorem}

\begin{proposition}\label{p00}
Let $S$ be the Sorgenfrey line. Then any uncountable subspace of $S$ is not a Lindel\"{o}f $\Sigma$-subspace.
\end{proposition}

\begin{proof}
Let $Y$ be an uncountable subspace of $S$. Assume that $Y$ is a Lindel\"{o}f $\Sigma$-subspace. Then $L=Y\cup (-Y)$ is also a Lindel\"{o}f $\Sigma$-subspace of $S$, hence $L^{2}$ has a countable network by Theorem~\ref{t125} since $S$ has a $G_{\delta}$-diagonal. However, $L^{2}$ contains a closed, uncountable and discrete subspace $\{(x, -x): x\in L\}$, which is a contradiction since $L^{2}$ has a countable network.
\end{proof}

\begin{example}\label{echarming spaces}
There exists a hereditarily Lindel\"{o}f space $X$ which is not a charming space.
\end{example}

\begin{proof}
Let $X$ be the Sorgenfrey line. Then $X$ is a hereditarily Lindel\"{o}f space. Assume that $X$ is a charming space. Then there exists a subspace $Y$ such that $Y$ is a $(\mathscr{P}_{4}, \mathscr{P}_{4})$-shell in $X$. By Proposition~\ref{p00}, $Y$ is countable. Let $Y=\{b_{n}: n\in\mathbb{N}\}$. For each $n\in\mathbb{N}$, take an open neighborhood $[b_{n}, b_{n}+2^{-n})$ in $X$. Then $U=\bigcup _{n\in\mathbb{N}}[b_{n}, b_{n}+2^{-n})$ is an open neighborhood of $Y$ in $X$. Since $m(Y)\leq \sum_{n=1}^{\infty}2^{-n}=1$, we have $X\setminus U$ is an uncountable set, then $X\setminus U$ is not a Lindel\"{o}f $\Sigma$-space by Proposition~\ref{p00}, which is a contradiction.
\end{proof}

It is natural to ask the following question.

\begin{question}\label{q0002}
Does each charming space $X$ with a $G_{\delta}$-diagonal have a countable network?
\end{question}

By Theorem~\ref{t125}, the following proposition is obvious.

\begin{proposition}
Each charming space $X$ with a $G_{\delta}$-diagonal is a $(2, 2)$-structured space.
\end{proposition}

The following two results are well-known in the class of Lindel\"{o}f $\Sigma$-spaces.

\begin{theorem}\cite{H1975}
Each hereditarily Lindel\"{o}f $\Sigma$-space has a countable network.
\end{theorem}

\begin{theorem}\cite{Gr}\label{tL}
Each Lindel\"{o}f $\Sigma$-space with a point-countable base is second-countable.
\end{theorem}

However, the following questions are remain open.

\begin{question}
Is each hereditarily charming space a Lindel\"{o}f $\Sigma$-space?
\end{question}

\begin{question}
Does each hereditarily charming space have a countable network?
\end{question}

\begin{question}
Is each charming space with a point-countable base metrizable?
\end{question}

\section{CL-charming and CO-charming spaces}
In this section, we shall introduce two special classes of charming spaces, and then discuss some topological properties of them.

\begin{definition}
A space $X$ will be called {\it CL-charming} (resp. \it CO-charming) if there is a closed subspace (resp. \it compact subspace) $Y$ of $X$ such that $Y$ a $(\mathscr{P}_{4}, \mathscr{P}_{4})$-shell in $X$.
\end{definition}

Obviously, each CO-charming space is CL-charming, and each CL-charming is a charming space. The space $X$ in Example~\ref{e1} is CO-charming. However, there exists a CL-charming space is not a CO-charming space.

\begin{example}\label{e2}
There exists a CL-charming space which is not CO-charming.
\end{example}

\begin{proof}
For each $n\in\mathbb{N}$, let $X_{n}$ be the copy of $X$ in Example~\ref{e1}. Then it is easy to see that the topological sum $Y=\bigoplus_{n\in\mathbb{N}}X_{n}$ is a CL-charming space. However, it is not a CO-charming space.
\end{proof}

The following three propositions are easy to check.

\begin{proposition}
Any image of a CO-charming space under a continuous
map is a CO-charming space.
\end{proposition}

\begin{proposition}
Any closed image of a CL-charming space under a continuous
map is a CL-charming space.
\end{proposition}

\begin{proposition}
Any preimage of a CL-charming space (resp.  CO-charming space) under a perfect
map is a CL-charming space (resp.  CO-charming space).
\end{proposition}

However, we do not know the answer of the following question.

\begin{question}
Is any image of a CL-charming space under a continuous
map a CL-charming space?
\end{question}

\begin{theorem}\label{ta}
A hereditarily CL-charming space is a hereditarily Lindel\"{o}f $\Sigma$-space. Therefore, $X$ has a countable network.
\end{theorem}

\begin{proof}
Let $X$ be a hereditarily CL-charming space. Take a closed subspace $Y$ of $X$ such that $Y$ is a Lindel\"{o}f $\Sigma$-space. Obviously, $X$ is a hereditarily Lindel\"{o}f space. Hence $Y$ has a countable pseudocharacter in $X$. Put $Y=\bigcap_{k\in\omega}U_{k}$, where each $U_{k}$ is an open neighborhood in $X$. Since every $X\setminus U_{k}$ is a Lindel\"{o}f $\Sigma$-space, we see that $\bigcup_{n\in\omega}(X\setminus U_{n})$ is a Lindel\"{o}f $\Sigma$-space. Therefore, $$\bigcup_{n\in\omega}(X\setminus U_{n})=X\setminus\bigcap_{n\in\omega}U_{n}=X\setminus Y$$ is a Lindel\"{o}f $\Sigma$-space. Then $X=(X\setminus Y)\cup Y$ is a Lindel\"{o}f $\Sigma$-space.
Therefore, $X$ is a hereditarily Lindel\"{o}f $\Sigma$-space, which implies that $X$ has a countable network \cite[corollary 4.13]{H1975}.
\end{proof}

\begin{corollary}
A hereditarily CO-charming space is a hereditarily Lindel\"{o}f $\Sigma$-space.
\end{corollary}

\begin{theorem}\label{tp}
A CO-charming space with a point-countable base is metrizable.
\end{theorem}

\begin{proof}
Let $X$ be a charming space with a a point-countable base $\mathcal{B}$. Since $X$ is CO-charming, there exists a compact subspace $K$ of $X$ such that $Y$ is a $(\mathscr{P}_{4}, \mathscr{P}_{4})$-shell in $X$. Moreover, since a compact space with a point-countable base is metrizable, $K$ is a separable and  metrizable space. Put $\mathcal{B}'=\{B\in \mathcal{B}: B\cap K\neq \emptyset\}$. Then $\mathcal{B}'$ is countable since $K$ is separable. Let $$\mathcal{U}=\{\bigcup \mathcal{F}: \mathcal{F}\subset \mathcal{B}', K\subset \bigcup \mathcal{F}, |\mathcal{F}|< \omega\}.$$ Then it is easy to check that $\mathcal {U}=\{U_n: n\in\mathbb{N}\}$ is a countable base of $K$ in $X$.  For each $n\in\mathbb{N}$, let $X_n=X\setminus U_n$. Since each $X_n$ is a Lindel\"{o}f $\Sigma$-subspace with a point-countable base, each $X_n$ is separable and metrizable by Theorem~\ref{tL}. Therefore, $$X=K\cup (X\setminus K)=K\cup (X\setminus \bigcap_{n\in \mathbb{N}}U_{n})=K\cup\bigcup_{n\in \mathbb{N}}(X\setminus U_{n}),$$ which implies that $X$ is separable. Since a separable space with a point-countable base has a countable base, $X$ is metrizable.
\end{proof}

\begin{question}
Is each CL-charming space $X$ with a point-countable base metrizable?
\end{question}

The following result gives a partial answer to Question~\ref{q0002}.

\begin{theorem}\label{tp}
A CO-charming space with a $G_{\delta}$-diagonal has a countable network.
\end{theorem}

\begin{proof}
Let $X$ be a CO-charming space. Then there exists a compact subspace $K$ of $X$ such that $K$ is a $(\mathscr{P}_{4}, \mathscr{P}_{4})$-shell in $X$. Since $X$ has a $G_{\delta}$-diagonal, $K$ is a $G_{\delta}$-set in $X$ by \cite[Proposition 2.3]{H2005}. Put $K=\bigcap_{n\in\omega}U_{n}$, where each $U_{n}$ is open in $X$. For each $n\in\mathbb{N}$, let $X_{n}=X\setminus U_{n}$. Then each $X_{n}$ is a Lindel\"{o}f $\Sigma$-subspace in $X$. Hence $$X\setminus K=X\setminus(\bigcap_{n\in\omega}U_{n})=\bigcup_{n\in\omega}X\setminus U_{n}=\bigcup_{n\in\omega}X_{n}$$ is a Lindel\"{o}f $\Sigma$-subspace. Therefore, $X=K\cup(X\setminus K)$ is a Lindel\"{o}f $\Sigma$-space. By \cite[Theorem 4]{T2010}, $X$ has a countable network.
\end{proof}

\medskip

\section{The Characterizations of weak $(i,j)$-structured spaces}
In this section, we shall give a characterization for some special class of charming spaces.

\begin{definition}
For any $i, j\in\{0, 1, 2, 3, 4, 5\}$, a space $X$ is called a {\it weak $(i, j)$-structured space} if there exists a space $Y$ which maps continuously onto $X$ and perfectly onto an $(i, j)$-structured space.
\end{definition}

Obviously, each weak $(1, 1)$-structured space is a charming space, hence it is a Lindel\"{o}f space. However, the following question is still open.

\begin{question}
Is each charming space weak $(1, 1)$-structured?
\end{question}

\begin{proposition}\label{p2}
Each closed subspace $F$ of a weak $(i, j)$-structured space $X$ is also a weakly $(i, j)$-structured space.
\end{proposition}

\begin{proof}
Take a space $Y$ for which there exists a continuous onto map $\varphi: Y\rightarrow X$ and a perfect map $h: Y\rightarrow M$ for some $(i,j)$-structured space $M$. Let $F$ be a closed subspace of $Y$  and let $Z=\varphi^{-1}(F)$. Then $F$ is a continuous image of $Z$ and it is easy to see that $h|Z: Z\rightarrow M$ is a perfect map. Hence $F$ is a weak $(i,j)$-structured space.
\end{proof}

The following theorem gives a characterization of weak $(i,j)$-structured spaces. First, we recall some concepts.

Any map $\varphi$ from a space $X$ to the family exp($Y$) of subsets of $Y$ is called {\it multivalued}; for convenient, we always write $\varphi$: $X \rightarrow Y$ instead of $\varphi$: $X \rightarrow exp(Y )$. A multivalued map $\varphi$: $X\rightarrow Y$ is called {\it compact-valued} (resp. {\it finite-valued}) if the set $\varphi(x)$ is compact (resp. finite) for each $x \in X$.

Let $\varphi$: $X\rightarrow Y$ be a multivalued map. For any $A \subset X$, we denotes the set $\bigcup\{\varphi(x) : x\in A\}$ by $\varphi(A)$, that is, $\varphi(A)= \bigcup\{\varphi(x) : x\in A\}$; we say that the map $\varphi$ is onto if  $\varphi(X) = Y $. The map $\varphi$ is called {\it upper semicontinuous} if $\varphi^{-1}(U) =\{x\in X :\varphi(x)\subset U\}$ is open in $X$ for any open subset $U$ in $Y$.

\begin{theorem}
For arbitrary $i, j\in\{0, 4, 5\}$, the following conditions are equivalent for any space $X$:
\begin{enumerate}
\item $X$ is a weak $(i,j)$-structured space;

\item there exists spaces $K$ and $M$ such that $K$ is compact, $M$ is an $(i, j)$-structured space and $X$ is a continuous images of a closed subspace of $K\times M$;

\item $X$ belongs to any class $\mathscr{P}$ which satisfies the following conditions:

(a) $\mathscr{P}$ contains compact spaces and $(i, j)$-structured spaces;

(b) $\mathscr{P}$ is invariant under closed subspaces and continuous images;

(c) $\mathscr{P}_{5}\times \mathscr{P}$ are contained in $\mathscr{P}$.

\item there is an upper semicontinuous compact-valued onto map $\varphi: M\rightarrow X$ for some $(i, j)$-structured space $M$.
\end{enumerate}
\end{theorem}

\begin{proof}
(1)$\Rightarrow$(2). There exists a space $Y$ which maps continuously onto $X$ and perfectly onto an $(i, j)$-structured space $M$. Then we fix the respective perfect map $h: Y\rightarrow M$. Let $i: Y\rightarrow \beta Y$ be the identity embedding. Then the diagonal map $g=h\Delta i: Y\rightarrow M\times \beta Y$ is perfect by \cite[Theorem 3.7.9]{E1989}, and so the set $g(Y)$ is closed in $M\times \beta Y$. Since $g$ is injective, $g(Y)$ is homeomorphic to $Y$, and hence $X$ is a continuous image of $g(Y)$.

(2)$\Rightarrow$(3). Suppose that $X$ is the continuous image of a closed subset $F$ of the product $M\times K$ for some $(i, j)$-structured space $M$ and a compact space $K$. Take any class $\mathscr{P}$ as in (2). Then $K\in\mathscr{P}$ and $M\in\mathscr{P}$ . Therefore $M\times K\in\mathscr{P}$ and hence $F\in\mathscr{P}$ as well since $\mathscr{P}$ is invariant under closed subspaces. Since $\mathscr{P}$ is invariant under continuous images, we have $X\in\mathscr{P}$.

(3)$\Rightarrow$(1). First, we will show if $X$ satisfies (1) then any closed subspace $F$ of $X$ also satisfies (1). Indeed, take a space $Y$ for which there exists a continuous onto map $\varphi: Y\rightarrow X$ and a perfect map $h: Y\rightarrow M$ for some $(i, j)$-structured space $M$. Let $F$ be an any closed subspace of $X$  and let $Z=\varphi^{-1}(F)$. Then $F$ is a continuous image of $Z$ and it is easy to see that $h|Z: Z\rightarrow M$ is a perfect map. Hence $F$ satisfies (1). It is evident that the class of spaces satisfies (1) is invariant under continuous images. Moreover, the classes of $(i, j)$-structured spaces and compact spaces satisfy (1).

(1)$\Rightarrow$(4). Suppose that $X$ satisfies (1), i.e., there exists a space $Y$ which maps continuously onto $X$ and perfectly onto an $(i, j)$-structured space $N$. Fix the respective map $f: Y\rightarrow X$ and a perfect map $g: Y\rightarrow N$. Let $M=g(Y)$. For every $x\in X$, since the set $\varphi(x)=f(g^{-1}(x))\subset X$ is compact, $\varphi: M\rightarrow X$ is a compact-valued map. The map $f$ is surjective, and hence $\varphi(M)=X$. It is easy to see that $\varphi$ is upper semicontinuous.

(4)$\Rightarrow$(2). Fix an $(i, j)$-structured space $M$ and a compact-valued upper semicontinuous onto map  $\varphi: M\rightarrow X$. Put $F=\bigcup\{\varphi(x)\times\{x\}: x\in M\}$. Then $F$ is contained in $\beta X\times M$. Let $\pi: \beta X\times M\rightarrow\beta X$ be the projection; then $\pi(F)=X$ so $X$ is a continuous image of $F$. Fix any point $(x, t)\in (\beta X\times M)\setminus F$. Then $x\not\in\varphi(t)$, and hence we can find disjoint sets $U, V\in\tau(\beta X)$ such that $x\in U$ and $\varphi(t)\subset V$. Since $\varphi$ is upper semicontinuou, there exists a set $W\in\tau(t, M)$ such that $\varphi(W)\subset V$. It easily check that $(x, t)\in U\times W\subset (\beta X\times M)\setminus F$; therefore the set $F$ is closed in $\beta X\times M$.
\end{proof}

\medskip

\section{Rectifiable spaces with a weak $(i,j)$-structure}
Recall that a {\it topological group} $G$ is a group $G$ with a
(Hausdorff) topology such that the product map of $G \times G$ into
$G$ is jointly continuous and the inverse map of $G$ onto itself
associating $x^{-1}$ with arbitrary $x\in G$ is continuous. A
topological space $G$ is said to be a {\it rectifiable space}
provided that there are a surjective homeomorphism $\varphi :G\times
G\rightarrow G\times G$ and an element $e\in G$ such that
$\pi_{1}\circ \varphi =\pi_{1}$ and for every $x\in G$ we have
$\varphi (x, x)=(x, e)$, where $\pi_{1}: G\times G\rightarrow G$ is
the projection to the first coordinate. If $G$ is a rectifiable
space, then $\varphi$ is called a {\it rectification} on $G$. It is
well known that rectifiable spaces are
a good generalizations of topological groups. In fact, for a
topological group with the neutral element $e$, then it is easy to
see that the map $\varphi (x, y)=(x, x^{-1}y)$ is a rectification on
$G$. However, the 7-dimensional sphere $S_{7}$ is
rectifiable but not a topological group \cite[$\S$ 3]{V1990}. The following theorem plays an important role in the study of rectifiable spaces.

\begin{theorem}\cite{C1992, G1996, V1989}\label{t124}
A topological space $G$ is rectifiable if and only if there exist $e\in G$ and two
continuous maps $p: G^{2}\rightarrow G$, $q: G^{2}\rightarrow G$
such that for any $x\in G, y\in G$ the next
identities hold:
$$p(x, q(x, y))=q(x, p(x, y))=y\ \mbox{and}\ q(x, x)=e.$$
\end{theorem}

Given a rectification $\varphi$ of the space $G$, we may obtain the maps $p$ and $q$ in Theorem~\ref{t124} as follows. Let $p=\pi_{2}\circ\varphi^{-1}$ and
$q=\pi_{2}\circ\varphi$. Then the maps $p$ and $q$ satisfy the identities in
Theorem~\ref{t124}, and both are open maps.

Let $G$ be a rectifiable space, and let $p$ be the multiplication on
$G$. Further, we sometimes write $x\cdot y$ instead of $p(x, y)$ and
$A\cdot B$ instead of $p(A, B)$ for any $A, B\subset G$. Therefore,
$q(x, y)$ is an element such that $x\cdot q(x, y)=y$; since $x\cdot
e=x\cdot q(x, x)=x$ and $x\cdot q(x, e)=e$, it follows that $e$ is a right neutral
element for $G$ and $q(x, e)$ is a right inverse for $x$. Hence a
rectifiable space $G$ is a topological algebraic system with
operations $p$ and $q$, a 0-ary operation $e$, and identities as above. It is
easy to see that this algebraic system need not satisfy the
associative law about the multiplication operation $p$. 

\begin{lemma}\label{l0}
If $A$ is a subset of a rectifiable space $G$, then $H=\bigcup_{n\in\mathbb{N}}(A_{n}\cup B_{n})$ is also a rectifiable subspace of $G$, where $A_{1}=A, B_{1}=q(A, e)\cup q(A, A), A_{2}=p(A_{1}\cup B_{1}, A_{1}\cup B_{1}), B_{2}=q(A_{1}\cup B_{1}, A_{1}\cup B_{1})$, $A_{n+1}=p(A_{n}\cup B_{n}, A_{n}\cup B_{n}), B_{n+1}=q(A_{n}\cup B_{n}, A_{n}\cup B_{n}), n=1, 2, \cdots.$ Obviously, if $A$ is a Lindel\"{o}f $\Sigma$-subspace, then $H$ is a Lindel\"{o}f $\Sigma$-space.
\end{lemma}

\begin{proof}
Since $q(A, A)\subset B_{1}$, we have $e\in B_{1}$. Therefore, it is easy to see that $A_{n}\cup B_{n}\subset A_{n+1}\cup B_{n+1}$ for each $n\in\mathbb{N}$. Put $H=\bigcup_{n\in\mathbb{N}}(A_{n}\cup B_{n})$. Next we shall prove that $H$ is a rectifiable subspace of $G$. Indeed, take arbitrary points $x, y\in B$. Then there exists an $n\in \mathbb{N}$ such that $x, y\in A_{n}\cup B_{n}$, and hence $p(x, y)\in A_{n+1}\cup B_{n+1}$ and $q(x, y)\in A_{n+1}\cup B_{n+1}$. Therefore, $H$ is a rectifiable subspace of $G$.

Since the product of a countable family of  Lindel\"{o}f $\Sigma$-space is Lindel\"{o}f $\Sigma$, it is easy to see that $H$ is Lindel\"{o}f $\Sigma$.
\end{proof}

\begin{theorem}\label{t0}
Every charming rectifiable space $G$ has a dense rectifiable subspace that is a Lindel\"{o}f $\Sigma$-space.
\end{theorem}

\begin{proof}
Since $G$ is charming, there exists a subspace $B$ such that $G$ is a $(\mathscr{P}_{4}, \mathscr{P}_{4})$-shell of $G$.

{\bf Case 1}: The set $B$ is dense in $G$.

Take the smallest rectifiable subspace $H$ of $G$ such that $L\subset G$. Then $H$ is also a Lindel\"{o}f $\Sigma$-space by Lemma~\ref{l0}. Clearly, $H$ is dense in $G$.

{\bf Case 2}: The set $B$ is not dense in $G$.

Obviously, there exists a non-empty open subset $U$ of $G$ such that $\overline{U}\cap B=\emptyset$. Put $V=G\setminus \overline{U}$ is an open neighborhood of $B$. Then $\overline{U}$ is a Lindel\"{o}f $\Sigma$-space. Hence $\mathscr{A}=\{x\cdot \overline{U}: x\in G\}$ is an open cover of $G$, and each element of $\mathscr{A}$ is homeomorphic to $\overline{U}$. Since $G$ is Lindel\"{o}f, there exists a countable subcover of $\mathscr{A}$, and hence $G$ is a Lindel\"{o}f $\Sigma$-space.
\end{proof}

A topological space $X$ has the {\it Suslin property} if every pairwise disjoint family of non-empty open subsets of $X$ is countable.

\begin{lemma}\cite{V1982}\label{l1}
The Suslin number of an arbitrary Lindel\"{o}f $\Sigma$-rectifiable space $G$ is countable.
\end{lemma}

\begin{theorem}
The Suslin number of an arbitrary charming rectifiable space $G$ is countable.
\end{theorem}

\begin{proof}
By Theorem~\ref{t0}, $G$ has a dense rectifiable subspace $H$ which is a Lindel\"{o}f $\Sigma$-subspace. Then the Suslin number of $H$ is countable by Lemma~\ref{l1}. Since $H$ is dense in $G$, it follows that the Suslin number of $G$ is countable.
\end{proof}

\begin{corollary}
The Suslin number of an arbitrary rectifiable space with a weak $(i,j)$-structure is countable.
\end{corollary}

\end{document}